\documentclass[11pt,reqno]{amsart}
\usepackage[utf8]{inputenc}
\usepackage{amsmath,amsthm,amssymb,enumerate}
\usepackage{geometry}
\newtheorem{theorem}{Theorem}[section]
\theoremstyle{plain}
\newtheorem{corollary}[theorem]{Corollary}
\newtheorem{definition}{Definition}[section]
\newtheorem{example}[theorem]{Example}
\newtheorem{lemma}[theorem]{Lemma}
\newtheorem{proposition}[theorem]{Proposition}
\numberwithin{equation}{section}

\newcommand{\C}{\mathbb{C}}
\newcommand{\Hperp}{\mathcal{H}^{\perp}}
\newcommand{\vol}{\mathrm{vol}}

\begin{document}
\title[Eigenvalues of the Kohn Laplacian]{Eigenvalues of the Kohn Laplacian and deformations of pseudohermitian structures on CR manifolds}

\author{Amine Aribi}
\address{Institut Denis Poisson, Universit\'{e}  de Tours, Universit\'{e}  d'Orl\'{e}ans, CNRS (UMR 7013), Parc de Grandmont, 37200 Tours, France\newline ESME, Paris, France, 34 Rue de Fleurus, 75006 Paris}
\email{Amine.Aribi@lmpt.univ-tours.fr; \quad amine.aribi@esme.fr }

\author{Duong Ngoc Son}
\address{Faculty of Fundamental Sciences, PHENIKAA University, Yen Nghia, Ha Dong, Hanoi 12116, Vietnam}
\email{son.duongngoc@phenikaa-uni.edu.vn}
\thanks{This project begun when the second-named author was at University of Vienna. He was supported by the Austrian Science Fund, FWF-Projekt M 2472-N35.}
\date{October 6, 2022}

\begin{abstract}
We study the eigenvalues of the Kohn Laplacian on a closed embedded strictly pseudoconvex CR manifold as functionals on the set of positive oriented pseudohermitian structures~$\mathcal{P}_+$. We show that the functionals are continuous with respect to a natural topology on~$\mathcal{P}_+$. Using an adaptation of the standard Kato--Rellich perturbation theory, we prove that the functionals are (one-sided) differentiable along 1-parameter analytic deformations. We use this differentiability to define the notion of critical pseudohermitian structures, in a generalized sense, for them. We give a necessary (also sufficient in some situations) condition for a pseudohermitian structure to be critical. Finally, we present explicit examples of critical pseudohermitian structures on both homogeneous and non-homogeneous CR manifolds.
\end{abstract}
\maketitle

\section{Introduction}\label{sec:intro}
Let $(M^{2n+1},\theta)$ be a compact strictly pseudoconvex pseudohermitian manifold, $\bar{\partial}_b$ the tangential Cauchy-Riemann operator,
and $\bar{\partial}_b^{*}$ the adjoint with respect to the volume form $d\vol_\theta:=\theta\wedge (d\theta)^n$. The Kohn Laplacian acting on functions is defined by $\Box_b = \bar{\partial}_b ^* \bar{\partial}_b$. It is well-known that  $\square_{b}$
is nonnegative and self-adjoint with noncompact resolvent on the Hilbert space $L^{2}\left(
M, d\vol_{\theta}\right) $ of the complex-valued square-integrable functions on $M$. Here, the inner product on $L^2(M,d\vol_{\theta})$ is defined by $\left\langle \phi , \psi \right\rangle :=\int_{M} \phi\, \overline{\psi}\, d\vol_\theta$.  This operator plays an important role in many problems in several complex variables and CR geometry, see, e.g. \cite{CS} and \cite{DT}. In particular, its spectrum contains rich geometric information about the underlying CR manifolds (see, e.g. \cite{BE,LSW} and the references therein).

The spectral theory for the Kohn Laplacian in the strictly pseudoconvex case is well understood. It is proved by Beals--Greiner \cite{BG} for the case $n\geq 2$ and Burns--Epstein \cite{BE} for the case $n=1$ that the spectrum of $\square_{b}$ in $\left(
0,\infty\right) $ consists of point eigenvalues of finite multiplicities (the case $n\geq 2$ is even simpler, thanks to Kohn's Hodge theory, see, e.g, \cite{CS,LSW}). Moreover, by Kohn \cite{Kohn}, $M$ is embeddable if and only if zero is an \textit{isolated} eigenvalue of $\Box_b$. Thus, if $M$ is embeddable, then $\mathrm{spec}\left( \square_{b}\right) \cap (0,+\infty) $ consists of countably many eigenvalues of finite multiplicities,
$0<\lambda_{1}\leq \lambda_{2} \leq \cdots$, with $\lambda_{j}\rightarrow\infty$ as $j\rightarrow\infty$. Moreover, 
for $j\geq 1$, the corresponding eigenfunctions are smooth. By the work of Boutet de Monvel \cite{BdM}, the embeddability holds for \textit{compact} strictly pseudoconvex CR manifolds if $n\geq 2$.

In recent years, there is much effort devoted to the study of the first positive eigenvalue $\lambda_1$ of the Kohn Laplacian. In particular, estimates for $\lambda_1$ have been studied extensively, see \cite{LSW,DLL,LS18} and the references therein. In the present paper, we consider, for each $k\geq 1$, the $k$-th eigenvalue $\lambda_k(\theta) = \lambda_k(\Box_b^{\theta})$ 
as a functional on the space of positive pseudohermitian structures $\mathcal{P}_+: = \{e^u\theta \colon u \in C^{\infty}(M) \}$ and study its behavior under deformations of the pseudohermitian structures. This study is motivated by previous work about spectral theory in Riemannian and CR geometries; see e.g. \cite{AhSa,ADS,ADS2}.

The first result of this paper establishes the continuity of eigenvalue functionals with respect to deformations of the contact forms. Precisely, fix a reference structure $\theta_0$ on $\mathcal{P}_{+}$ and consider the $C^1$-distance on $\mathcal{P}_{+}$ given by
\begin{equation}\label{e:distancedefn}
d_{0}(\theta, \theta')
=
\sup_M |u - u'| + \sup_M |\bar{\partial}_b u - \bar{\partial}_b u'|_{\theta_0},
\end{equation}
where $\theta = e^u \theta_0$ and $\theta' = e^{u'} \theta_0$. The continuity of the $\lambda_k$-functionals is stated as follows.
\begin{theorem}\label{thm:main}
Let $M^{2n+1}$ be a compact strictly pseudoconvex embeddable CR manifold. Suppose that $\theta$ and $\hat{\theta} = e^{u}\theta$ are two pseudohermitian structures on $M$. For $\delta>0$ and $\delta'>0$,
if $\sup_M |u| < \delta$ and $\sup_M |\bar{\partial}_b u|_{\theta} < \delta'$, then 
\begin{align}\label{e:son}
e^{-(n-1/2)\delta}\sqrt{\lambda_k(\theta)} - n \delta'e^{n\delta} \leq \sqrt{\lambda_k(\hat{\theta})}\leq e^{(n-1/2)\delta}\sqrt{\lambda_k(\theta)} + n \delta'e^{(n-1/2)\delta}.
\end{align}
In particular, the map $\theta \mapsto \lambda_k(\theta)$ is locally Lipschitz continuous on $(\mathcal{P}_{+}, d_0)$.
\end{theorem}
The proof is based on an analogue of the ``Max-mini principle'' for the eigenvalues of the Kohn Laplacian (see \cite{BU, ADS} for the (sub-)Laplacian counterparts). A new difficulty that arises in our situation is the fact that the kernel $\ker (\Box_b)$ is nontrivial. In fact, $\ker(\Box_b)$ consists of CR functions and has infinite dimension. We overcome this difficulty by restricting $\Box_b$ to the orthogonal complement of its kernel. We point out, however, that the orthogonality also depends on the pseudohermitian structure.

An immediate application of Theorem~\ref{thm:main} is the semi-continuity of the multiplicities of the eigenvalues. Precisely, let $m_k(\theta)$ be the multiplicity of the eigenvalue $\lambda_k(\theta)$, i.e., 
\begin{equation}
m_k(\theta)
:=
\# \{\ell \mid  \lambda_{\ell}(\theta)=\lambda_k(\theta)\}.
\end{equation}
Adapting the proof of Corollary 2.12 in \cite{urakawa2017}, we obtain from Theorem~\ref{thm:main} the following corollary.
\begin{corollary}\label{cor:semicont}
Let $(M,\theta)$ be an embeddable strictly pseudoconvex pseudohermitian manifold. Then there exists $\delta > 0$ such that whenever $\hat{\theta}\in \mathcal{P}_{+}$ with $d_{0}(\hat{\theta},\theta) <\delta$, then
\begin{equation}
m_k(\hat{\theta}) \leq m_k(\theta).
\end{equation}
\end{corollary}

Subsequent results of this paper establish the one-sided differentiability of $\lambda_k$-functionals and the criticality of pseudohermitian structures with respect to 1-parameter (smooth or analytic) deformations. Namely, let $e^{u_t} \theta$ be an analytic deformation of the pseudohermtian structure. For each $k\geq 1$, the function $t \mapsto \lambda_k(\theta_t)$ is differentiable at almost every $t$, but it may fail to be differentiable at certain points. However, by an adaptation of the perturbation theory for unbounded self-adjoint operators with compact resolvent of Rellich--Alekseevsky--Kriegl--Losik--Michor, see F. Rellich \cite{FR}, D. Alekseevski \& A. Kriegl \& M. Losik \& P.W. Michor \cite{AKL}, and A. Kriegl \& P.W. Michor \cite{KM03}, we prove that the function $t \mapsto \lambda_k(\theta_t)$ admits left-sided and right-sided derivatives at $t=0$. The left and right derivatives can be expressed in terms of the
eigenvalues of the Hermitian form defined as follows: Let
\begin{equation}\label{e:Ldef}
L(\psi, \eta) = (n+1) \bar{\eta}\, \Box_{b}\psi - n\langle \bar{\partial}_b \psi , \partial_b \bar{\eta}\rangle,
\end{equation}
and
\begin{equation}\label{e:Qdef}
Q_f(\psi , \eta) = -\int_M fL(\psi , \eta)\, d\vol_\theta. 
\end{equation}
Then the restriction $L|_{E_k}$ to each eigenspace $E_k$ is Hermitian, i.e., $L(\psi , \eta) = \overline{L(\eta, \psi)}$ for $\eta, \psi \in E_k$. Moreover, if $f$ is real-valued, then $Q_f$ is also Hermitian. Therefore, $Q_f|_{E_k}$ has $m := \dim {E_k}$ real eigenvalues, counting multiplicities.
Our next result is as follows.
\begin{theorem}\label{thm:1.2}
Let $(M,\theta)$ be an embeddable strictly pseudoconvex pseudohermitian manifold and $\theta(t) = e^{u_t}\theta,\,t\in (-\epsilon,\epsilon),$ an analytic deformation, $\theta(0) = \theta$. For each $k\geq 1$, let $\lambda_k(\theta(t))$ be the $k$-th eigenvalue of $\Box_{b,t}$. Then
\begin{enumerate}[(i)]
\item The function $t \mapsto \lambda_{k}(\theta(t))$ has left and right derivatives at $t=0$.
\item The one-side derivatives $\frac{d}{dt} \lambda_{k}(\theta(t))\big|_{t=0^-}$ and $\frac{d}{dt}\lambda_{k}(\theta(t))\big|_{t=0^+}$ are eigenvalues of the Hermitian form $Q_{f}|_{E_k}$, where $f = \partial u_t/\partial t |_{t=0}$.
\item If $k=1$ or $\lambda_{k}(\theta ) > \lambda_{k-1}(\theta )$, then $\frac{d}{dt}\lambda_{k}(\theta(t))\big|_{t=0^-}$ and $\frac{d}{dt}\lambda_{k}(\theta(t))\big|_{t=0^+}$ are the greatest and the least eigenvalues of $Q_{f}|_{E_k}$, respectively.
\item If $\lambda_{k}(\theta) < \lambda_{k+1}(\theta )$ then
$\frac{d}{dt}\lambda_{k}(\theta(t))\big|_{t=0^-}$ and $\frac{d}{dt}\lambda_{k}(\theta(t))\big|_{t=0^+}$
are the smallest and the greatest eigenvalue of $Q_{f}|_{E_k}$, respectively.
\end{enumerate}
\end{theorem}
This theorem should be compared to similar results for Laplacian \cite{AhSa} and sub-Laplacian on CR manifolds \cite{ADS2}. In view of this theorem, we define the notion of critical pseudohermitian structures for the $\lambda_k$-functional as follows: We first denote by $\mathcal{P}_{+}^0$ the space of strictly pseudoconvex pseudohermitian structures with unit volume, i.e.,
\begin{equation}
{\mathcal P}_{+}^0=\left\{\theta \in {\mathcal P}_+(M): \int_M\,d\vol_\theta=1\right\}.
\end{equation}
We say that a pseudohermitian structure $\theta$ is \textit{critical} for the $\lambda_{k}$-functional restricted to $\mathcal P_{+}^0$ if for any analytic deformation $\left\{\theta (t)= e^{u_t}\theta\right\}\subset \mathcal{P}_{+}^0$ with $\theta(0)=\theta$, we have
\begin{equation}
\frac{d}{dt}\lambda_{k}(\theta(t))\big|_{t=0^-} \times \frac{d}{dt}\lambda_{k}(\theta(t))\big|_{t=0^+} \leq 0.
\end{equation}
If $\theta$ is critical for the $\lambda_{k}$-functional, then for any analytic deformation of unit volume $t\to \theta(t)$ (i.e. $\theta(t) \in \mathcal{P}^0_+$ for all $t$), either
\begin{equation}
\lambda_k(\theta(t))\leq \lambda_k(\theta)+o(t)\quad \text{as} \ \ t\to 0,
\end{equation}
or 
\begin{equation}
\lambda_k(\theta(t))\geq\lambda_k(\theta)+o(t)\quad \text{as}\ \ t \to 0.
\end{equation}
Observe that if $k=1$ then only the first possibility can occur.

In the next result, we give a characterization of the criticality for the $\lambda_k$-functional.
\begin{theorem}\label{thm:main2}
Let $M$ be an embeddable strictly pseudoconvex CR manifold and $\theta \in \mathcal{P}^0_{+}$. If $\theta$ is critical for the $\lambda_k$-functional restricted to $\mathcal{P}^0_{+}$, then there exists a finite family $\psi_1, \dots, \psi_d$ of eigenfunctions corresponding to $\lambda_k$ such that 
\begin{equation}\label{e:l1}
\sum_{j=1}^d L(\psi_j) = \sum_{j=1}^{d} \left((n+1)\lambda_k|\psi_j|^2 - n |\bar{\partial}_b \psi_j|^2 \right) = \mathrm{constant}.
\end{equation}
Here $L$ is defined by \eqref{e:Ldef}. If $k=1$ or $\lambda_{k-1}(\theta) < \lambda_{k+1}(\theta)$, then the existence of such a family of eigenfunctions is also sufficient for $\theta$ to be critical for the $\lambda_k$-functional.
\end{theorem}
We should point out that although the characterization \eqref{e:l1} is similar to the Riemannian case \cite{AhSa} and sub-Riemannian case \cite{ADS2} (see also \cite{AJK} for a similar result in K\"ahler case), our case exhibits an important difference. Precisely, it is proved in \cite{AhSa} that the criticality of the $\lambda_k(\Delta)$-functional of the Laplacian is characterized by the existence of a finite collection of $\lambda_k(\Delta)$-eigenfunctions $f_j$ such that the sum of squares $f_1^2 + f_2^2 + \dots + f_d^2$ is constant on the manifold. In our characterization, the identity \eqref{e:l1} contains not only the sum of squared norms of the eigenfunctions, but also their first-order derivatives. It is natural to ask whether the term involving derivatives in \eqref{e:l1} can be removed? It is worth noting that in the examples of critical pseudohermitian structures given in Section~\ref{sec:examples}, there always exist collections of eigenfunctions whose sums of squared norms are constant. However, a difficulty in our situation comes from the fact that $\Box_b$ is generally a complex operator and hence the eigenfunctions are \textit{not} necessarily real-valued. This makes the method treating the Laplacian \cite{AhSa} and sub-Laplacian \cite{ADS} cases break down in the Kohn Laplacian case. Nevertheless, the characterization \eqref{e:l1} still leads to the following corollary.

\begin{corollary}\label{cor:mul}
Let $(M,\theta)$ be a compact embeddable pseudohermitian manifold. Suppose that $(M, \theta)$ is homogeneous (i.e., the group of CR diffeomorphisms preserving $\theta$ acts transitively). Then $\theta$ is critical for the $\lambda_k$-functional if either $k=1$, or $k>1$ and $\lambda_{k-1} < \lambda_{k+1}$.
\end{corollary}

The paper is organized as follows. In Section~\ref{sec:pre}, we study the continuity of eigenvalue functionals and prove Theorem~\ref{thm:main} and Corollary~\ref{cor:semicont}. In Section~\ref{sec:der}, we study parametrizations of the eigenvalues using the classical perturbation theory that is adapted to our situation. We study the critical pseudohermitian structures and prove Theorem~\ref{thm:main2} and Corollary~\ref{cor:mul} in Section~\ref{sec:proofs}. In Section~\ref{sec:ratios}, we extend some results for $\lambda_k$-functionals to the case of ratio functionals $\frac{\lambda_{k+1}}{\lambda_k}$ of two consecutive eigenvalues and give
characterizations of critical pseudohermitian structures for them. In Section~\ref{sec:examples}, we give several explicit examples of both homogeneous and nonhomogeneous critical structures.
\section{Continuity of the eigenvalue functionals}\label{sec:pre}
\subsection{The Kohn Laplacian on pseudohermitian manifolds}
We briefly recall some basic notions of pseudohermitian geometry and the Kohn Laplacian on pseudohermitian manifolds. For more details, we refer the readers to \cite{DT} and \cite{CS}. Let $(M,\theta)$ be a strictly pseudoconvex pseudohermitian manifold. Let $T$ be the Reeb field associated to $\theta$, i.e., $T$ is the unique real vector field that satisfies $T \rfloor d\theta = 0$ and $\theta(T) = 1$. An admissible coframe on an open subset of $M$
is a collection of $n$ complex $(1,0)$-forms $\theta^1, \theta^2,\dots \theta^n$ whose restrictions to 
$T^{1,0}M$ form a basis for $(T^{1,0}M)^{\ast}$ and $\theta^{\alpha} (T) = 0$. There exists
a holomorphic frame $\{Z_{\alpha}\colon \alpha = 1,2,\dots n\}$ of $T^{1,0}M$ such that
$\{T, Z_{\alpha}, Z_{\bar{\alpha}}\}$ is the dual frame for $\{\theta, \theta^{\alpha}, \theta^{\bar{\alpha}}\}$.
The Levi form associated to $\theta$ is given by the Hermitian matrix $h_{\alpha\bar{\beta}}$, where
\begin{equation}
d\theta = i h_{\alpha\bar{\beta}} \theta^{\alpha} \wedge \theta^{\bar{\beta}}.
\end{equation}
Let $\partial_b$ be the Cauchy--Riemann operator. For a smooth function $f$, it holds that $\partial_b f = f_{\alpha}\theta^{\alpha}$
(summation convention) where $f_\alpha = Z_{\alpha}f$. The formal adjoint of $\partial_b$ on functions (with respect to the Levi form
and the volume element $d\vol_{\theta}:=\theta\wedge (d\theta)^n$) is given by $\partial_b^{*} = -\delta_b$. Here $\delta_b$ is the divergence operator taking $(1,0)$-forms to functions by $\delta_b(\sigma_{\alpha}\theta^{\alpha}) = \sigma_{\alpha,}{}^{\alpha}$. Here, a Greek index preceded by a comma indicates the covariant derivative with respect to the Tanaka--Webster connection on $(M,\theta)$.
The Kohn Laplacian associated to $\theta$ acting on functions is 
$\Box_b = \bar{\partial}_b ^* \bar{\partial}_b $, 
where $\bar{\partial}_b$ is the conjugate of $\partial_b$. In terms of the Tanaka--Webster covariant derivatives (see \cite{DT}), 
\begin{equation}
\Box_b f 
=
- f_{\bar{\alpha},}{}^{\bar{\alpha}}.
\end{equation}
If $M$ is compact and embeddable, then $\ker \Box_b$ consists of the CR functions and is of infinite dimension. In particular, $\ker \Box_b$ does \textit{not} depend on the choice of the pseudohermitian structure. As already mentioned in the introduction, the basic spectral theory for the Kohn Laplacian on compact embeddable strictly pseudoconvex pseudohermitian manifolds are well understood, see, e.g., \cite{CS}.

We shall need a formula relating the Kohn Laplacian operators associated to different pseudohermitian structures. Observe that if $\hat{\theta} = e^{u}\theta$, then $|\bar{\partial}_b f|_{\hat{\theta}} = e^{-u/2}|\bar{\partial}_b f|_{\theta}$ and
\begin{equation}
d\vol_{\hat{\theta}} = \hat{\theta} \wedge (d\hat{\theta})^n = e^{(n+1)u} d\vol_{\theta}.
\end{equation}
Furthermore, the Kohn Laplacian changes as follows:
\begin{proposition}\label{prop:boxtrans}
Let $(M,\theta)$ be a pseudohermitian manifold and let $\hat{\theta} = e^{u}\theta$. Denote by $\Box_b$ and $\hat{\Box}_b$ the Kohn Laplacian operators that correspond to $\theta$ and $\hat{\theta}$, respectively. Then
\begin{equation}\label{e:boxtrans}
e^u\,\hat{\Box}_{b} f = \Box_b f - n\,\langle \partial_b u , \bar{\partial}_b f\rangle.
\end{equation}
\end{proposition}
\begin{proof} We can assume that $f$ is smooth. Choose a local holomorphic frame $\{Z_{\alpha}\}$ and its dual coframe $\{\theta^{\alpha}\}$ that is admissible for $\theta$. If $\hat{\theta} = e^{u}\theta$, then we can take $\hat{Z}_{\alpha} = e^{-u/2} Z_{\alpha}$ and its dual coframe $\hat{\theta}^{\alpha} = e^{u/2}(\theta^{\alpha} + \sqrt{-1}u^{\alpha}\theta)$. Then the corresponding Levi matrices satisfy $\hat{h}_{\beta\bar{\gamma}} = h_{\beta\bar{\gamma}}$. Furthermore, the connection forms $\omega_{\beta}{}^{\alpha}$ and $\hat{\omega}_{\beta}{}^{\alpha}$ satisfy \cite{Lee},
\begin{align}\label{e:cftrans}
\hat{\omega}_{\beta}{}^{\alpha}
& =
\omega_{\beta}{}^{\alpha}
+
(u_{\beta}\theta^{\alpha} - u^{\alpha} \theta_{\beta})
+
\frac{1}{2}\delta_{\beta}^{\alpha}(u_{\mu}\theta^{\mu} - u^{\mu}\theta_{\mu}) \notag \\
& \quad +
\frac{\sqrt{-1}}{2}\left(u^{\alpha}{}_{\beta} + u_{\beta}{}^{\alpha} + 2u_{\beta}u^{\alpha} + 2\delta_{\beta}^{\alpha} u_{\mu}u^{\mu}\right) \theta.
\end{align}
Denote $\hat{\nabla}_{\beta} = \hat{\nabla}_{\hat{Z}_{\beta}}$, $\nabla_{\beta} = \nabla_{Z_{\beta}}$ the Tanaka--Webster covariant differentiation with respect to $\hat{\theta}$ and $\theta$ in the corresponding frames $\hat{Z}_{\beta}$ and $Z_{\beta}$. Then, by \eqref{e:cftrans},
\begin{align}
\hat{\nabla}_{\beta} \hat{\nabla}_{\bar{\gamma}} f
=
\hat{Z}_{\beta} \hat{Z}_{\bar{\gamma}} f - \overline{\hat{\omega}_{\gamma}{}^{\alpha}(\hat{Z}_{\bar{\beta}})}\, \hat{Z}_{\bar{\alpha}} f 
=
e^{-u} \nabla_{\beta} \nabla_{\bar{\gamma}} f + e^{-u} u^{\bar{\alpha}} f_{\bar{\alpha}} h_{\beta\bar{\gamma}}.
\end{align}
Taking the trace with respect to $\hat{h}_{\beta\bar{\gamma}}$, we obtain \eqref{e:boxtrans}. The proof is complete.
\end{proof}
\subsection{The Max-mini principle}\label{sec:con}
In this section, we prove an analogue for the Kohn Laplacian of the well-known Max-mini principle for eigenvalues of the (sub-)Laplacian in \cite{BU} and \cite{ADS}. This is an important ingredient for our proof of Theorem~\ref{thm:main}. Precisely, for each $k$-dimensional complex subspace
$L_{k} \subset C^{\infty}(M,\C)$ that satisfies $L_k \cap \ker \Box_b = \{0\}$, we put
\begin{equation}
\Lambda_{\theta}(L_k) = \sup \left\{\frac{\|\Box_b f\|^2}{\|\bar{\partial}_{b} f\|^2}\colon f \in L_k, \ f\ne 0\right\}.
\end{equation}
\begin{lemma}[Max-mini Principle]
Let $(M,\theta)$ be a compact embeddable strictly pseudoconvex pseudohermitian manifold. Then
\begin{equation}
\lambda_{k}(\theta) = \inf_{L_k} \Lambda_{\theta}(L_k),
\end{equation}
where the infimum is taken over all subspaces $L_k \subset C^{\infty}(M,\C)$ of complex dimension~$k$ that satisfies $L_k \cap \ker \Box_b = \{0\}$.
\end{lemma}
\begin{proof} 
Let $\{e_j\}_{j=1}^{\infty}$ be a complete orthonormal system of (smooth) eigenfunctions of $(\ker\Box_b) ^{\perp}$, with $\Box_b e_j = \lambda_j e_j$. For arbitrary $f\in L^2(M,d\vol_{\theta})$, we expand
\begin{equation}
f = f_0+\sum_{j=1}^\infty a_j(f) e_j,
\end{equation}
where $f_0$ is CR holomorphic, $a_j(f) \in \C$, and the series is convergent in $L^2(M,d\vol_{\theta})$. Let $L_{k}^0$ be the subspace of $C^{\infty}(M,\C)$ spanned by $e_1,\dots e_k$. 
If $f\in L_{k}^0$, then 
\begin{equation}
f = \sum_{j=1}^k a_j(f) e_j.
\end{equation}
Thus, $\Box_b f = \sum_{j=1}^k a_j(f) \lambda_j e_j$, and hence,
\begin{align}
\|\Box_{b} f\|^2 
& = \sum_{j=1}^k | a_j(f)|^2 \lambda_j^2 \notag \\ 
& \leq \sum_{j=1}^k | a_j(f)|^2 \lambda_j\lambda_k \notag \\
& = \lambda_k \|\bar{\partial}_{b} f\|^2 .
\end{align}
Consequently,
\begin{equation}\label{e:21}
\Lambda_{\theta}(L_k^0) \leq \lambda_k(\theta).
\end{equation}
Here we have used an argument that is somewhat similar to those in \cite{BU} and \cite{ADS}.

To prove the reverse inequality, we shall also adapt the usual argument as appeared \cite{ADS, BU}. First observe that the case $k=1$ is immediate and well-known (see, e.g., Corollary~3.2 in \cite{DLL}). Thus, suppose that $k\geq 2$ and assume, for a contradiction, that there exists a subspace $L_k$ of complex dimension $k$ in $C^{\infty}(M,\C) $ that satisfies 
$L_k \cap \ker \Box_b = \{0\}$ and 
\begin{equation}
\Lambda_{\theta}(L_k) < \lambda_{k}.
\end{equation}
For each $f\in L_k$, 
\begin{align}
\Lambda_{\theta}(L_k) \sum_{j=0}^{\infty} |a_j(f)|^2\lambda_j
& = \Lambda_{\theta}(L_k)\cdot (f , \Box_b f)_{L^2} \notag \\
& \geq \|\Box_b f\|^2 \notag \\
& = \sum_{j=0}^{\infty} \lambda_j^2 |a_j(f)|^2.
\end{align}
Thus, for any $m\geq 1$,
\begin{equation}\label{e:11}
\sum_{j=1}^{m} |a_j(f)|^2 \lambda_j\left( \Lambda_{\theta}(L_k) - \lambda_j\right)
\geq
\sum_{j=m+1}^{\infty} |a_j(f)|^2 \lambda_j \left( \lambda_j - \Lambda_{\theta}(L_k) \right).
\end{equation}
In particular, choose $ m = \max\{ j \geq 0 \colon \lambda_j(\theta) \leq \Lambda_{\theta} (L_k)\}$.
Observe that $0\leq m \leq k-1$ since $\lambda_k > \Lambda_{\theta}(L_k)$.
Let $\Phi \colon L_k \to C^{\infty}(M, \C)$ be the linear map given by
\begin{equation}
\Phi(f) = \sum_{j = 1}^{m} a_j(f) e_j.
\end{equation}
Let $\widetilde{f} \in \ker \Phi$, then $a_j(\widetilde{f}) = 0$ for all $j = 1,2,\dots, m$.
Applying \eqref{e:11} to $\widetilde{f}$ yields 
\begin{equation}
\sum_{j=m+1}^{\infty} |a_j(\widetilde{f})|^2 \lambda_j \left( \lambda_j - \Lambda_{\theta}(L_k) \right) \leq 0.
\end{equation}
Since $\lambda_j - \Lambda_{\theta}(L_k) >0 $ for all $j \geq m+1$ we deduce that
$a_j(f_0) = 0$ for all $j=1,2,\dots$. Thus, $\ker \Phi = \{0\}$, i.e, $\Phi$ is injective. This contradicts the fact that 
the image $\Phi(L_k)$ has dimension at most $m < k = \dim L_k$. The proof is complete. 
\end{proof}
\subsection{Proof of Theorem~\ref{thm:main} and Corollary~\ref{cor:semicont}}
\begin{proof}[Proof of Theorem~\ref{thm:main}] Let $L_k$ be a $k$-dimensional subspace of $L^2(M, \theta)$ such that $L_k \cap \ker \Box_b = \{0\}$. 
For each $f\in L_k$, $f\ne 0$, we define two quotients $Q$ and $\widehat{Q}$ as follows.
\begin{equation}
Q= \frac{ \|\Box_b f\|^2_{\theta}}{\left\|\bar{\partial}_b f \right\|^2_{\theta}} \quad
\text{and} \quad \widehat{Q} = \frac{\|\widehat{\Box}_b f\|^2_{\hat{\theta}}}{\left\|\bar{\partial}_b f \right\|^2_{\hat{\theta}}}.
\end{equation}
By \eqref{e:boxtrans}, we have
\begin{align}
|\widehat{\Box}_b f|^2
& =
e^{-2u}|\Box_b f - n \langle \partial_b u , \bar{\partial}_b f\rangle |^2 \notag \\
& = 
e^{-2u}\left[|\Box_b f|^2 + n^2 |\langle \partial_b u , \bar{\partial}_b f\rangle |^2 - 2n \Re \left((\Box_b f) \overline{\langle \partial_b u , \bar{\partial}_b f\rangle}\right)\right].
\end{align}
We deduce, since $d\vol_{\hat{\theta}} = e^{(n+1)u} d\vol_\theta$, that
\begin{align}
\|\widehat{\Box}_b f\|_{\hat{\theta}}^2
& = 
\int_M e^{(n-1)u}\left[|\Box_b f|^2 + n^2 |\langle \partial_b u , \bar{\partial}_b f\rangle |^2 - 2n \Re \left((\Box_b f) \overline{\langle \partial_b u , \bar{\partial}_b f\rangle}\right)\right] d\vol_\theta\notag \\
& \leq 
e^{(n-1)\delta} \left(\|\Box_b f\|_{\theta}^2 + n^2 \delta'^2 \|\bar{\partial}_b f\|^2 +2n \delta' \|\Box_ b f\| \cdot \|\bar{\partial}_b f\| \right).
\end{align}
We have used $|u| < \delta$, $ |\partial_b u|_{\theta} \leq \delta'$ on $M$, and the Cauchy--Schwarz inequality. On the other hand,
\begin{align}
\|\bar{\partial}_b f\|_{\hat{\theta}}^2
=
\int _M |\bar{\partial}_b f|_{\theta}^2\, e^{nu} d\vol_\theta  \geq 
e^{-n \delta} 	\|\bar{\partial}_b f\|_{\theta}^2.
\end{align}
We then deduce that
\begin{align}
\widehat{Q} 
& \leq 
e^{(2n-1) \delta}\left(\sqrt{Q} + n \delta'\right)^2.
\end{align}
From the Max-Mini principle, we deduce that
\begin{align}
\lambda_k(\hat{\theta}) \leq e^{(2n-1)\delta}\left(\sqrt{\lambda_k(\theta)} + n \delta'\right)^2.
\end{align}
This is the second inequality in \eqref{e:son}. To prove the first inequality, we exchange the roles of $\theta$ and $\hat{\theta}$. Observe
that $\hat{\theta} = e^{-u}\theta$ and $|\bar{\partial}_b u|_{\hat{\theta}} = e^{u/2} |\bar{\partial}_b u|_{\theta} < e^{\delta/2}\delta'.$ From the argument above,
\begin{equation}
\sqrt{\lambda_k(\theta)} \leq e^{(n-1/2)\delta} \left(\sqrt{\lambda_k(\hat{\theta)}} + n e^{\delta/2}\delta'\right),
\end{equation}
which clearly implies the first inequality in \eqref{e:son}. The proof is complete.
\end{proof}
\begin{proof}[Proof of Corollary~\ref{cor:semicont}]
We put $\lambda = \lambda_k(\theta)$ and $m = m_k(\theta)$. We first consider the case $\lambda_k(\theta) = \lambda_1(\theta)$. Thus, $1\leq k\leq m $ and 
\begin{equation}
\lambda_1(\theta) = \cdots  = \lambda_{m}(\theta) < \lambda_{m+1}(\theta) \leq \cdots 
\end{equation}
Put $\epsilon: = (1/2) (\lambda_{m+1}(\theta) - \lambda) >0$ and choose $\delta > 0$ and $\delta' > 0$ such that
\begin{equation}
\lambda_{m+1} - \epsilon
\leq 
\left(e^{-(n-1/2) \delta} \sqrt{\lambda_{m+1}(\theta)} - n\delta' e^{n\delta} \right)^2.
\end{equation}
Using Theorem~\ref{thm:main}, we deduce that that for $j\geq m+1$ and $\hat{\theta}: = e^u \theta$, 
\begin{align}
\lambda + \epsilon
=
\lambda_{m+1}(\theta) - \epsilon
& \leq 
\left(e^{-(n-1/2) \delta} \sqrt{\lambda_{m+1}(\theta)} - n\delta' e^{n\delta} \right)^2 \notag \\
& \leq 
\left(e^{-(n-1/2) \delta} \sqrt{\lambda_{j}(\theta)} - n\delta' e^{n\delta} \right)^2 \notag \\
& \leq \lambda_j(\hat{\theta}),
\end{align}
provided that $|u| < \delta$ and $|\partial_b u|_{\theta} < \delta'$ on $M$. Consequently, $m_k(\hat{\theta}) \leq m_k(\theta)$ for all $\hat{\theta}$ that is ``closed enough'' to $\theta$, as desired.

Next, we consider the case $\lambda_k(\theta) > \lambda_1(\theta)$. Then for some $\ell$ with $1\leq \ell < k \leq \ell +m$,
\begin{equation}
\lambda_{\ell}(\theta) < \lambda_{\ell+1}(\theta) = \cdots = \lambda_k(\theta) = \cdots = \lambda_{\ell+m}(\theta) < \lambda_{\ell + m +1} \leq \cdots  
\end{equation}
Arguing similarly as above, we can find $\epsilon>0, \delta > 0$, and $\delta' > 0$ such that whenever $\hat{\theta} = e^u\theta$ with $|u| < \delta$ and $|\partial_b u|_{\theta} < \delta'$ on $M$, it holds that
\begin{equation}
\lambda + \epsilon \leq \lambda_j(\hat{\theta}), \quad j \geq \ell + m +1,
\end{equation}
and
\begin{equation}
\lambda_j(\hat{\theta}) \leq \lambda - \epsilon, \quad j =1,2, \dots , \ell.
\end{equation}
Hence, $m_k(\hat{\theta}) \leq m_k(\theta)$. The proof is complete.
\end{proof}
\begin{corollary}
Let $t \mapsto u_t$ be a $C^{1,\alpha}$ curve of smooth functions on $M$, with $\alpha>0$, and $\theta(t) = e^{u_t}\theta$. Then the curve $t\mapsto \lambda_k(\theta(t))$ is differentiable almost everywhere.
\end{corollary}
\begin{proof}
Observe that $t\mapsto \lambda_k(\theta(t))$ is locally Lipschitz by Theorem~\ref{thm:main}, the conclusion then follows from the well-known Rademacher's theorem; see \cite[Theorem~7.20]{rudin1986}.
\end{proof}
\section{Derivatives of the eigenvalue functionals}\label{sec:der}
We make use of a perturbation result which is an adaptation of the well-known general theory due to Rellich \cite{FR}, Alekseevsky, Kriegl, 
Losik, and Michor \cite{AKL} for unbounded self-adjoint operators with compact resolvents. The literature regarding the perturbation theory for self-adjoint operators is vast and we cannot describe it in detail here. We refer the readers to, e.g., \cite{FR,KM03} and the references therein for a detailed account.
\begin{proposition}\label{thm:genrellich}
Let $t\mapsto A(t)$ be an analytic curve of (possibly unbounded) closed operators on a Hilbert space $H$, with a common domain of definition and with compact resolvents. Suppose that, for each $t$, the spectrum of $A(t)$ is a discrete set of positive \emph{real} eigenvalues with finite multiplicities. Suppose further that the eigenvectors of $A(t)$ form a basis for $H$ and that the global resolvent set $\{(t,z) \colon A(t) - z \text{ is invertible\,}\}$ is open. Then the eigenvalues and the eigenvectors of $A(t)$ may be  parametrized real analytically in $t$, locally.
\end{proposition}
An well-known argument in the literature (see, e.g., Theorem~7.8 in \cite{AKL} or \cite{KM03}) reduces the parametrizations of the eigenvalues of a family of self-adjoint operators to those of the real roots of hyperbolic polynomials. The proof sketched below uses an argument that is almost the same: The only difference is that instead of assuming the self-adjointness, we assume the operators $A(t)$ have ``nice'' spectra, namely, the eigenvalues of $A(t)$ are \textit{real} and they have ``enough'' eigenvectors. These assumptions guarantee that the finite family of the eigenvalues that are enclosed by a certain curve in the global resolvent set are the real roots of an analytically parametrized hyperbolic polynomial. Moreover, the direct sum of the corresponding eigenspaces admits an analytic framing. This implies that the eigenvectors are parametrized analytically, locally in $t$.

Precisely, let $V$ be the common domain of definition of $A(t)$ for all $t$. As in \cite{AKL}, put 
\begin{equation}
\|u\|_t^2 : = \|u\|^2 + \|A(t) u\|^2.
\end{equation}
Then, for each $t$, $(V, \|\cdot \|_t)$ is a Hilbert space and $A(t) \colon V\to H$ is bounded. Moreover, the norms are locally uniformly equivalent in $t$ \cite{AKL}. We equip $V$ with one of the norms, say $\|\cdot \|_{t_0}$. The map $(t,z)\mapsto (A(t) - z)^{-1}$ is analytic for $(t,z)$ in the global resolvent set (which is assumed to be open) since the inversion is analytic in the space $L(V,H)$.

Fix a parameter $s$ and choose a smooth and closed curve $\gamma$ in the resolvent set of $A(s)$ such that
for all $t$ that is close to $s$, there is no eigenvalue of $A(t)$ lie on $\gamma$. This is possible under the assumption that the global resolvent set is open. The curve 
\begin{equation}
t \mapsto P(t,\gamma) : = -\frac{1}{2i} \int_{\gamma} (A(t) - z)^{-1}dz
\end{equation}
is a smooth curve of projections (onto the direct sums of the eigenspaces that correspond to the eigenvalues of $A(t)$ interior to $\gamma$) with finite and constant rank, say $N$. This family
of $N$-dimensional complex vector spaces $t\mapsto P(t,\gamma) (H)$ admits a local analytic frame~$\mathcal{F}$.

The operators $A(t)$ map $P(t,\gamma) (H)$ into itself. In the local analytic frame~$\mathcal{F}$,
they are given by $N\times N$ matrices $M_t$ that are parametrized analytically in $t$. By assumptions, for each $t$, $M_t$ has $N$ real eigenvalues which are precisely the eigenvalues of $A(t)$ that are interior to $\gamma$. Moreover, the eigenvalues are the real roots of the (hyperbolic) characteristic polynomial $P_t$ of $M_t$. Consequently, the eigenvalues of $A(t)$ that are interior to $\gamma$ are parametrized analytically near~$s$. Since the frame of $P(t,\gamma)(H)$ can be chosen locally analytically, the corresponding eigenvectors can also be chosen real analytically, as desired.

\begin{theorem}\label{thm:dif}
Let $(M,\theta)$ be a compact embeddable strictly pseudoconvex pseudohermitian manifold. Let $\theta(t):=e^{u_t}\theta$ be an analytic deformation of $\theta$. Let $\lambda >0$ be an eigenvalue of $\Box_{b}$ with multiplicity $m$. Then there exist $\epsilon >0$, a family of $m$ real-valued analytic functions $\Lambda_{j}$ on $(-\epsilon, \epsilon)$, and a family of smooth functions $\phi_j(t)$ on $M$ such that
\begin{enumerate}[(i)]
\item $\Lambda_j(0) = \lambda$, $1\leq j \leq m$,
\item $\Box_{b,t} \phi_j(t) = \Lambda_{j}(t) \phi_j(t)$.
\end{enumerate}
\end{theorem}
\begin{proof}
Let $\Box_{b,t}$ be the Kohn Laplacian associated to $\theta(t)$. For each
$t$, $\Box_{b,t}$ is a self-adjoint operator in $L^2\left(M,d\vol_{\theta(t)}\right)$. Moreover, $\ker \Box_{b,t} = \mathcal{H}$ is the subspace of CR functions. Note that $\mathcal{H}$ does not depend on $t$ and is closed in $L^2\left(M,d\vol_{\theta(t)}\right)$ for every $t$. We denote by $\Hperp_t$ the orthogonal complement of $\mathcal{H}$ in $L^2\left(M,d\vol_{\theta(t)}\right)$:
\begin{equation}
L^2(M,\theta_t) = \mathcal{H} \oplus \Hperp_t.
\end{equation}
Note that this orthogonal decomposition depends on $t$. For each $t$, the restriction 
\[
\Box_{b,t}\bigl|_{\Hperp_t} \colon {\Hperp_t} \to {\Hperp_t}
\]
is an unbounded self-adjoint operator with compact resolvent (see, eg., \cite{LSW}). The spectrum of $\Box_{b,t}|_{\Hperp_t}$ coincides with the spectrum of $\Box_{b,t}$ with zero removed.

Consider the analytic family of  operators
\begin{align}
U_t \colon \mathcal{H}^{\perp} \to \mathcal{H}_t^{\perp},
\quad
U_t(\varphi)= e^{-(n+1)u_t} \varphi.
\end{align}
Define $P_t \colon {\Hperp_t} \to {\Hperp_t}$ by
\begin{equation}
P_t = U_t^{-1} \circ \left(\Box_{b,t}|_{\Hperp_t}\right) \circ U_t.
\end{equation}
Then $P_t$ is a family of (not necessary self-adjoint) operators with compact resolvents. Moreover, for each $t$, the spectrum of $P_t$ coincides with the spectrum of $\Box_{b,t}|_{\Hperp_t}$. In particular, $P_t$ has a discrete spectrum consisting of real eigenvalues of finite multiplicities.

Observe that $P_t$ is analytic in $t$. Indeed, by Proposition~\ref{prop:boxtrans}, we have
\begin{equation}
e^{u_t}P_t (\varphi)
=
\Box_b \varphi + \langle \partial_b u_t , \bar{\partial}_b \varphi \rangle
+(n+1) \langle \partial_b \varphi , \bar{\partial}_b u_t \rangle
-(n+1)\left(\Box_b u_t - (n-1)|\partial_b u_t|^2\right) \varphi,
\end{equation}
and the right-hand side depends analytically in $t$.

By the continuity of the eigenvalues proved in Theorem~\ref{thm:main}, the global resolvent set must be open. Therefore, we can apply Proposition~\ref{thm:genrellich} above to conclude that the eigenvalues of $P_t$ can be parametrized analytically on $t$, i.e., (i) holds. The conclusion (ii) also follows at once.
\end{proof}
\subsection*{Proof of Theorem~\ref{thm:1.2}}
Let $m$ be the dimension of $E_k(\theta)$, the eigenspace of $\Box_b$ with eigenvalue $\lambda_k(0)$. We apply Theorem~\ref{thm:dif} to obtain $m$ real-analytic functions $\Lambda_l(t)$ and $m$ analytic families of smooth functions $\phi_l(t) \colon M \to \mathbb{C}$, $|t|<\epsilon$ such that
\begin{equation}\label{e:37}
\Box_{b,t} \phi_l(t) = \Lambda_l(t) \phi_l(t), \quad l = 1,2,\dots , m.
\end{equation}
By the continuity of the eigenvalues proved in Theorem~\ref{thm:main}, we deduce that there are two indices $p$ and $q$ such that
\begin{equation}\label{e:eg}
\lambda_k(\theta(t))
=
\begin{cases}
\Lambda_p(t) ,\quad t < 0, \\
\Lambda_q(t) ,\quad t > 0.
\end{cases}
\end{equation}
From this, the existence of the left and right derivatives of $t\mapsto \lambda_k(\theta(t))$ follows immediately.

To prove (ii), we differentiate \eqref{e:37} to obtain
\begin{equation}
\Box_b ' \phi_l = - \Box_b \phi_l' + \Lambda_l'(0) \phi_l + \lambda_k(\theta) \phi_l'.
\end{equation}
By integration by parts, for all $\psi \in E_k(\theta)$,
\begin{align}
\langle \Box_b ' \phi_l , \psi \rangle_{L^{2}(\theta)}
& = 
-\int_M \bar{\psi} \, \Box_{b} \phi_l ' + \Lambda_l'(0) \int_M \phi_l\,\bar{\psi} + \lambda_k(0) \int_M \phi_l'\, \bar{\psi} \notag \\
& = 
-\int_M \phi_l ' \, \overline{\Box_{b} \psi} + \Lambda_l'(0) \int_M \phi_l\, \bar{\psi} + \lambda_k(0) \int_M \phi_l'\, \bar{\psi} \notag \\
& = \,
\left\langle \Lambda_l'(0) \phi_l , \psi \right \rangle_{L^{2}(\theta)}.
\end{align}
Therefore, $\Pi_k \circ \Box_b '|_{E_k}$ is diagonalizable with eigenvalues $\{\Lambda_l'(0)\}$. Here, $\Pi_k$ is the orthogonal projection onto $E_k$. We claim that the eigenvalues of $\Pi_k \circ \Box_b'|_{E_k}$ are those of $Q_f|_{E_k}$ and hence (ii) follows. Indeed, by Proposition~\ref{prop:boxtrans},
\begin{equation}
\Box_{b}' \varphi := -f\, \Box_b \varphi - n\langle \partial_b f , \bar{\partial}_b \varphi \rangle,
\end{equation}
where $f = \partial u_t/\partial t |_{t=0}$. On the other hand,
\begin{align}
\int_M \bar{\psi} \langle \bar{\partial}_b \varphi , \partial_b f\rangle
& =
\int_M f\bar{\partial}^{\ast}_b (\bar{\psi}\bar{\partial}_b \varphi) d\vol_\theta \notag \\
& = 
\int_M f\bar{\psi} (\bar{\partial}_b ^{\ast}\bar{\partial}_b \varphi) d\vol_{\theta}
-
\int_M f \langle \bar{\partial}_b \varphi , \partial_b \bar{\psi} \rangle \, d\vol_{\theta} \notag \\
& =
\int_M f\left(\bar{\psi} \Box_b \varphi - \langle \bar{\partial}_b \varphi , \partial_b \bar{\psi} \rangle \right) \, d\vol_{\theta}.
\end{align}
Therefore,
\begin{align}
\int_M \bar{\psi}\, \Box_b ' \varphi \, \, d\vol_{\theta}
=
-\int_M f\left((n+1)\bar{\psi} \Box_b \varphi - n\langle \bar{\partial}_b \varphi , \partial_b \bar{\psi} \rangle \right) \, d\vol_{\theta} 
=
Q_f(\varphi, \psi).
\end{align}
Hence, the claim and (ii) follow.

To prove (iii), we suppose that $\lambda_k(\theta) > \lambda_{k-1}(\theta)$. Since $\Lambda_l(0) = \lambda_{k}(\theta)$ and $\lambda_{k}(\theta(t))$ is continuous in $t$, it must hold that $\Lambda_l(t) > \lambda_{k-1}(\theta(t))$ for $t$ sufficiently small. This implies that for such $t$, 
\begin{equation}
\lambda_k(\theta(t)) = \min \{\Lambda_1(t), \Lambda_2(t), \dots , \Lambda_m(t)\}.
\end{equation}
Hence
\begin{equation}
\frac{d}{dt} \lambda_k(\theta(t))\bigl|_{t=0^{-}} 
=
\max \{\Lambda_1'(0), \Lambda_2'(0), \dots , \Lambda_m'(0) \}
\end{equation}
and 
\begin{equation}
\frac{d}{dt} \lambda_k(\theta(t))\bigl|_{t=0^{+}} 
=
\min \{\Lambda_1'(0), \Lambda_2'(0), \dots , \Lambda_m'(0) \}
\end{equation}
which prove (iii).

Part (iv), i.e., the case $\lambda_k(\theta) < \lambda_{k+1}(\theta)$, can be proved similarly. We omit the details.
\section{Proofs of Theorem~\ref{thm:main2} and Corollary~\ref{cor:mul}} \label{sec:proofs}
\subsection{Proof of Theorem~\ref{thm:main2}}
We denote by $\mathcal{A}_{0}(M,\theta)$ the set of real-valued smooth functions $f$ with zero mean on $M$: 
\begin{equation}
\mathcal{A}_{0}(M,\theta)= \bigg\{f\in C^\infty (M; \mathbb{R}): \int_{M} f \; d\vol_{\theta}=0 \bigg\}.
\end{equation}
We need the following result.
\begin{lemma}\label{thm:6.2} Let $M$ be a compact embeddable strictly pseudoconvex CR manifold and $\theta$ a pseudohermitian structure on $M$. Then the following statements hold.
\begin{enumerate}[(i)]
\item If $\theta$ is critical for the $\lambda_{k}$-functional restricted to $\mathcal P^0_{+}$, then for every $f \in \mathcal{A}_{0}(M,\theta)$, the restriction $Q_{f}|_{E_k(\theta)}\colon E_k(\theta) \rightarrow \mathbb{C}$ is indefinite.
\item Assume that either $k=1$ or $\lambda_{k-1}(\theta) < \lambda_{k+1}(\theta)$. Then $\theta$ is critical for the $\lambda_{k}$-functional restricted to $\mathcal P^0_{+}$ if and only if the Hermitian form $Q_{f}|_{E_k(\theta)}$ is indefinite for every $f \in \mathcal{A}_{0}(M,\theta)$.
\end{enumerate}
\end{lemma}
\begin{proof}
(i) Let $f \in \mathcal{A}_{0}(M,\theta)$. By direct calculations (cf. \cite{ADS2}, page 124), the pseudoconformal deformation of $\theta$ given by
\begin{equation}
\theta(t)=\left[\frac{\vol(\theta)}{\vol(e^{tf}\theta)}\right]^\frac{1}{n+1}e^{tf}\theta=e^{u_t}\theta, \quad t\in \mathbb{R},
\end{equation}
belongs to $\mathcal P^0_{+}$ and depends analytically on $t$ with $\frac{d}{dt} \theta(t)\big|_{t=0}= f \theta$, and
\begin{equation}
u_t=tf-\frac{1}{n+1}\ln(\vol(e^{tf}\,\theta)).
\end{equation}
Moreover,
\begin{equation}
\frac{d}{dt} u_t\big|_{t=0}=f.
\end{equation}
Now assuming $\theta$ is critical for $\lambda_k$ restricted to $\mathcal{P}^0_{+}$, using Theorem~\ref{thm:1.2}, we obtain that $Q_{f}|_{E_k}$ has both nonnegative and nonpositive eigenvalues and hence (i) follows.

(ii) Let $\theta(t)=e^{u_t}\theta\in \mathcal P^0_{+}$ be an analytic deformation of $\theta$. Since $\vol(\theta(t))$ is constant with respect to $t$, the function $f=\frac{d}{dt} u_t\big|_{t=0}\in \mathcal{A}_{0}(M,\theta).$ Indeed,
\begin{equation}
\frac{d}{dt}{\vol}(\theta(t))\big|_{t=0}=\frac{d}{dt} \int_M e^{(n+1)u_t}\vol_{\theta}\big|_{t=0}=(n+1)\int_M f\vol_{\theta}.
\end{equation}
Thus (ii) follows in view of Theorem \ref{thm:1.2}. 
\end{proof}
\begin{lemma}\label{lem:3}
The Hermitian form $Q_{f}|_{E_k}$ is indefinite on $E_{k}(\theta)$ for all $f\in \mathcal{A}_{0}(M,\theta)$ if and only if there exists a finite family $\{\psi_1, \psi_2, \dots , \psi_d\} \subset E_{k}(\theta)$ such that
\begin{equation}\label{e:l1a}
\sum_{j=1}^{d} L(\psi_j) = \sum_{j=1}^{d} \left((n+1)\lambda_k|\psi_j|^2 - n |\bar{\partial}_b \psi_j|^2 \right) = C
\end{equation}
for some constant $C >0$.
\end{lemma}
\begin{proof}
We use an argument that is by now standard (cf. \cite{ADS2}). First, assume that \eqref{e:l1} holds, then for any $f\in \mathcal{A}_{0}(M,\theta)$, it holds that
\begin{equation}
\sum_{j=1}^{d} Q_{f}(\psi_j,\psi_j) = - \sum_{j=1}^d\int_M fL(\psi_j)\, d\vol_\theta = - \int_M Cf\, d\vol_\theta = 0.
\end{equation}
Therefore, $Q_{f}|_{E_k(\theta)}$ must be indefinite.

Conversely, assume that $Q_{f}|_{E_k}$ is indefinite on $E_k(\theta)$ for all $f\in \mathcal{A}_{0}(M,\theta)$. For each $k$, let $\mathcal{C}_k$ be the convex set of $L^2(M,d\vol_{\theta})$ defined as follows.
\begin{equation}\label{e:ck}
\mathcal{C}_k:=
\left\{
\sum_{j\in J}L(v_j)
\colon
v_j \in E_k(\theta) , J \subset \mathbb{N}, \ J \ \text{finite}
\right\}.
\end{equation}
Using a stardard argument based on the classical separation theorem (cf. \cite{ADS2}), we can show that the constant $1$ belongs to $\mathcal{C}_k$. Indeed, if $1\not\in \mathcal{C}_k$, then we can find a smooth real-valued function $h$ such that
\begin{equation}
\int_M |h|^2 \, d\vol_\theta > 0
\end{equation}
and
\begin{equation}
\int_M h w \, d\vol_\theta \leq 0
\end{equation}
for all $w \in \mathcal{C}_k$. Let $f = h - h_0$, where $h_0>0$ is the average of $h$ on $M$. Then $f\in \mathcal{A}_0(M,\theta)$. For all $\psi \in E_k(\theta)$ we have, since $L(\psi) \in \mathcal{C}_k$,
\begin{align}
Q_{f}(\psi,\psi)
& = 
- \int_M L(\psi) (h-h_0) \, d\vol_\theta \notag \\
& = 
- \int _M h L(\psi) \, d\vol_\theta + h_0 \int _M \left((n+1)\lambda_k |\psi|^2 - n|\bar{\partial}_b \psi|^2\right) \, d\vol_\theta \notag \\
& \geq
h_0 \int _M \left((n+1)\lambda_k |\psi|^2 - n|\bar{\partial}_b \psi|^2\right) \, d\vol_\theta \notag \\
& =
h_0 (n+1)\int _M \bar{\psi} \Box_{b} \psi \, d\vol_{\theta}- h_0 n \int_M |\bar{\partial}_b \psi|^2 \, d\vol_\theta \notag \\
& =
h_0 \int _M |\bar{\partial}_b \psi|^2 \, d\vol_\theta.
\end{align}
This contradicts the assumption that $Q_{f}$ is indefinite on $E_{\theta}$ since the last integral is positive. Hence $1\in \mathcal{C}_k$ and thus there exists a family of functions $\{\psi_1, \psi_2, \dots , \psi_d\}$ that satisfies \eqref{e:l1}. Moreover, integrating \eqref{e:l1} and using integration by parts, we see that
\begin{equation}
C\vol(M,\theta) = \sum_{j=1}^{d} \int_M |\bar{\partial}_b \psi|^2 > 0,
\end{equation}
and hence $C>0$ as desired.
\end{proof}

Theorem \ref{thm:main2} follows immediately from the two lemmas above.

\begin{proof}[Proof of Corollary~\ref{cor:mul}]
Let a group $G$ act transitively on $M$ by pseudohermitian diffeomorphisms: $g^{\ast} \theta = \theta$ for each $g\in G$. Let $\{\psi_{1}, \psi_{2}, \dots , \psi_{d}\}$ be an orthonormal basis for $E_k$, then, for each $g\in G$, $\{g \cdot \psi_{j}: = \psi \circ g^{-1}\}$ is also an orthonormal basis and thus
\begin{equation}
g \cdot \psi_j = \sum_{k=1}^d a_{jk} \psi_k
\end{equation}
for some $d\times d$-unitary matrix $[a_{jk}]$. This implies that 
\begin{equation}
\Psi: = \sum_{j=1}^{d} |\psi_j|^2
\end{equation}
is $G$-invariant, and hence constant. By the same reason, $ \sum_{j=1}^{d} |\bar{\partial}_b\psi_j|^2$ is also a
constant. The proof then follows from Theorem~\ref{thm:main2}.
\end{proof}
\begin{corollary}\label{cor:mul2}
Suppose that $\theta$ is critical for the $\lambda_{k}$-functional. Then either $\lambda_{k}$ is a multiple eigenvalue or there exists a nontrivial eigenfunction $\psi$ such that $L(\psi)$ is a constant.
\end{corollary}
We have to leave open the question whether the latter case in the conclusion of Corollary~\ref{cor:mul} can happen.
\subsection{A refinement}
As briefly discussed in the introduction, our characterization of the criticality for pseudohermitian structures \eqref{e:l1} involves the first-order derivatives of the eigenfunctions. In this section, we show that under an additional condition, the term involving the derivatives can be removed. This follows from the lemma below.
\begin{lemma}\label{lem:4} Suppose that $\psi_1, \psi_2,\dots ,\psi_d \in E_k(\theta)$ satisfy
\begin{equation}\label{e:cri}
\sum_{j=1}^{d} \left((n+1)\lambda_k|\psi_j|^2 - n |\bar{\partial}_b \psi_j|^2 \right) = \lambda_k.
\end{equation}
Assume that 
\begin{equation}\label{e:wird}
\int_M\left|\sum_{j=1}^d \bar{\psi}_j \partial_b \psi_j \right|^2 d\vol_{\theta}
\leq 
\int_M\left|\sum_{j=1}^d \bar{\psi}_j \bar{\partial}_b \psi_j \right|^2 d\vol_{\theta}.
\end{equation}
then 
\begin{equation}
\sum_{j=1}^{d} |\psi_j|^2 = \mathrm{constant.}
\end{equation}
\end{lemma}
We point out that condition \eqref{e:wird} holds if either $\psi_j$'s are real-valued for all $j$, or the conjugations $\overline{\psi}_j$ are CR for all~$j$.
\begin{proof}
Let $\varphi = \sum_{j=1}^{d} |\psi_j|^2$. Integrating both sides of \eqref{e:cri} and using integration by parts, we have (we drop the volume form to simplify our notations)
\begin{align}
\lambda_k \vol(M) 
& = (n+1)\lambda_k \int_M \varphi - n \sum_{j=1}^d \int_M |\bar{\partial}_b \psi_j|^2 \notag \\
& = 
(n+1)\lambda_k \int_M \varphi - n \sum_{j=1}^d \int_M \lambda_k |\psi_j|^2 \notag \\
& = \lambda_k \int_M \varphi.
\end{align}
It follows that $\int_M \varphi = \vol(M)$. On the other hand, by direct calculations, 
\begin{equation}
\Box_b \varphi 
=
\lambda_k \varphi + \sum_{j=1}^d \psi_j \Box_b \bar{\psi}_j - \sum_{j=1}^d |\partial_b \psi_j|^2 - \sum_{j=1}^d |\bar{\partial}_b \psi_j|^2.
\end{equation}
This and \eqref{e:cri} imply that
\begin{align}\label{e:ha}
\int_M \varphi \Box_b \varphi
=
\frac{\lambda_k}{n} \int_M \varphi(1 - \varphi) + \sum_{j=1}^d \int_M \varphi \psi_j \Box_b \bar{\psi}_j - \sum_{j=1}^d \int_M \varphi |\partial_b \psi_j|^2.
\end{align}
Using integration by parts, we obtain for every $j$ and $\ell$, 
\begin{align}
\int_M \psi_j \bar{\psi}_{j,\bar{\alpha}} \psi_{\ell} \bar{\psi}_{\ell,\alpha}
& =
- \int_M \bar{\psi}_j \left(\psi_j \psi_{\ell} \bar{\psi}_{\ell,\alpha} \right)_{,\bar{\alpha}} \notag \\
& = - \int_M \bar{\psi}_j\psi_{\ell} \psi_{j,\bar{\alpha}}\bar{\psi}_{\ell,\alpha} 
- \int_M |\psi_j|^2 |\bar{\partial}_b \psi_{\ell}|^2
+ \lambda_k\int_M |\psi_j|^2 |\psi_{\ell}|^2.
\end{align}
Here, the Greek indexes preceded by commas denote the Tanaka--Webster covariant derivatives with respect to an orthonormal frame of $T^{1,0}M$ and their conjugates. Taking the sum over $j$ and $\ell$, we obtain
\begin{align}
\sum_{j,\ell} \int_M \psi_j \bar{\psi}_{j,\bar{\alpha}} \psi_{\ell} \bar{\psi}_{\ell,\alpha}
& =
-\int_M \left|\sum_{j=1}^d \bar{\psi}_j \bar{\partial}_b \psi_j \right|^2
- \sum_{j=1}^d \int_M \varphi |\bar{\partial}_b \psi_j|^2 + \lambda_k \int_M \varphi^2 \notag \\
& = 
-\int_M \left|\sum_{j=1}^d \bar{\psi}_j \bar{\partial}_b \psi_j \right|^2 + \frac{\lambda_k}{n}\int_M \varphi(1-\varphi).
\end{align}
We compute,
\begin{align}
\sum_{j=1}^d \int_M \varphi \psi_j \Box_b \bar{\psi}_j - \sum_{j=1}^d \int_M \varphi |\partial_b \psi_j|^2
& = 
\sum_{j=1}^d \int_M \bar{\psi}_{j,\bar{\alpha}}(\varphi\psi_j)_{,\alpha} - \sum_{j=1}^d \int_M \varphi |\partial_b \psi_j|^2 \notag \\
& = 
\sum_{j=1}^d \int_M\psi_j \bar{\psi}_{j, \bar{\alpha}} \varphi_{\alpha} \notag \\
& = 
\sum_{j,\ell=1}^d \int_M \psi_j \bar{\psi}_{j,\bar{\alpha}} \psi_{\ell} \bar{\psi}_{\ell,\alpha} 
+
\int_M \left|\sum_{j=1}^d \bar{\psi}_j \partial_b \psi_j \right|^2 \notag \\
& = 
\int_M \left|\sum_{j=1}^d \bar{\psi}_j \partial_b \psi_j \right|^2
- \int_M \left|\sum_{j=1}^d \bar{\psi}_j \bar{\partial}_b \psi_j \right|^2 \notag \\
& \quad + \frac{\lambda_k}{n}\int_M \varphi(1-\varphi).
\end{align}
Plugging this into \eqref{e:ha}, we have that 
\begin{align}
0
& \leq \int_M (\Box_b \varphi) \varphi \notag \\
& =
\frac{2\lambda_k}{n}\int_M \varphi (1-\varphi)   + \int_M 
\left|\sum_{j=1}^d \bar{\psi}_j \partial_b \psi_j \right|^2
-\int_M\left|\sum_{j=1}^d \bar{\psi}_j \bar{\partial}_b \psi_j \right|^2 \notag \\
& \leq 0.
\end{align}
In the last inequality, we have used the fact that the average value of $\varphi$ is 1. Therefore, $\varphi$ must be a constant. The proof is complete.
\end{proof}
More generally, \eqref{e:wird} holds if $\psi_j$'s satisfy the following Beltrami-type equation for CR quasiconformal mappings (see \cite{KR}),
\begin{equation}\label{e:qcm}
f_{\alpha} = \mu_{\alpha}{}^{\bar{\beta}} f_{\bar{\beta}},
\end{equation}
almost everywhere on $M$, where $\mu = \mu_{\alpha}{}^{\bar{\beta}}$ is a tensor field whose operator norm (viewed as a field of complex linear mappings from $(T^{0,1}M)^\ast \to (T^{1,0}M)^\ast$) is less than one. We thus obtain the following corollary.
\begin{corollary}\label{prop:crsphere}
Let $(M,\theta)$ be a compact embeddable strictly pseudoconvex pseudohermitian manifold and let $\lambda_k$ be an eigenvalue of the Kohn Laplacian. Assume that $k=1$ or $\lambda_{k-1} (\theta) < \lambda_{k+1}(\theta)$. 
Assume further that the corresponding eigenfunctions satisfy the Beltrami-type equation \eqref{e:qcm} with $\|\mu\| \leq 1$, or has a basis consisting of real-valued functions. Then the following are equivalent.
\begin{enumerate}[(i)]
\item $\theta$ is a critical for the $\lambda_k$-functional.
\item There exists a finite family of eigenfunctions $\{\psi_1, \psi_2, \dots , \psi_d\}$ such that $|\psi_1|^2 + |\psi_2|^2 + \cdots + |\psi_d|^2 = 1$ on $M$.
\end{enumerate}
\end{corollary}
\section{Eigenvalue ratio functionals}\label{sec:ratios}
We consider the scaling invariant eigenvalue ratio functionals $ \theta \mapsto \displaystyle{{\lambda_{k+1}(\theta) / \lambda_{k}(\theta)}}$ for $k \geq 1$. If $\theta(t)$ is any analytic deformation of a pseudohermitian structure $\theta$, then by Theorem~\ref{thm:dif}, $t \mapsto \displaystyle{{\lambda_{k+1}(\theta(t)) / \lambda_{k}(\theta(t))}}$ admits left and right derivatives at $t=0$. Therefore, we can introduce the following notion (cf. \cite{AhSa,ADS2}).
\begin{definition} A pseudohermitian structure $\theta$ is said to be critical for the ratio $\lambda_{k+1} / \lambda_{k}$ if for any analytic deformation $\theta(t) = e^{u_t}\theta$, the left and right derivatives of $\lambda_{k+1}(\theta (t)) / \lambda_{k}(\theta(t))$ at $t=0$ have opposite signs or one of them vanishes.
\end{definition}
We introduce, for each $\theta\in \mathcal{P}_+$, the operator $P_k:E_{k}(\theta) \otimes E_{k+1}(\theta)\rightarrow E_{k}(\theta) \otimes E_{k+1}(\theta)$ defined by
\begin{equation}
P_k=\lambda_{k+1}(\theta) (\Pi_k\circ \Box'_b)\otimes I_{E_{k+1}(\theta)}-\lambda_{k}(\theta) I_{E_{k}(\theta)}\otimes (\Pi_{k+1}\circ \Box'_b),
\end{equation}
where $\Pi_k:L^2 (M, \Psi_\theta )\rightarrow E_k(\theta)$ is the orthogonal projection and $I$ is the identity. The Hermitian form naturally associated with $P_{k}$, denoted by $\widetilde{Q}_{f}$, is defined as follows: For every $v_1,\,w_1 \in E_{k}(\theta)$ and $v_2,\, w_2 \in E_{k+1}(\theta)$,
\begin{equation}
\widetilde{Q}_{f}\big(v_1\otimes w_1, v_2 \otimes w_2) = \langle \Box_b w_1, w_2\rangle_{L^{2}(\theta)} Q_{f}(v_1,v_2)
- \langle \Box_b v_1 , v_2 \rangle_{L^{2}(\theta)} Q_{f}(w_1,w_2).
\end{equation}
\begin{theorem}\label{theo 61}
Let $(M,\theta)$ be a compact embeddable strictly pseudoconvex pseudohermtian manifold. Then $\theta$ is critical for the functional $\frac{\lambda_{k+1}}{\lambda_{k}}$ if and only if the Hermitian form $ \widetilde{Q}_f$ is indefinite on $ E_{k}(\theta) \otimes E_{k+1}(\theta)$ for every real-valued regular function $f$.
\end{theorem}
\begin{proof} Firstly, we consider the case $\lambda_{k+1}(\theta)=\lambda_{k}(\theta)$. The ratio functional $\lambda_{k+1}/\lambda_k$ attains a global minimum at $ \theta$ and hence $\theta$ must be critical for the ratio functional. On the other hand, since $ \widetilde{Q_f} (v_1\otimes v_2, v_1\otimes v_2)=0$ for all $(v_1,v_2) \in E_k(\theta) \times E_{k+1}(\theta)$, the Hermitian form $\widetilde{Q_f}$ is indefinite on $ E_{k}(\theta) \otimes E_{k+1}(\theta)$, as desired.

Secondly, assume that $\lambda_{k+1}(\theta)>\lambda_{k}(\theta)$ and let $\theta( t)$ be an analytic deformation of $\theta$. From Theorem \ref{thm:1.2}, we have
\begin{equation}
\frac{d}{dt}\lambda_{k}(\theta(t))\big|_{t=0^-}\quad \text{and} \quad \frac{d}{dt}\lambda_{k}(\theta(t))\big|_{t=0^+}
\end{equation}
are the least and the greatest eigenvalues of $(\Pi_k\circ \Box'_b)$ on $E_{k}(\theta)$ respectively, and similarly for $k+1$. Therefore,
\begin{equation}
\lambda_{k}(\theta)^{2} \frac{d}{dt}\frac{\lambda_{k+1}(\theta(t))}{\lambda_{k}(\theta(t))}\Big|_{t=0^-}
=\left[\lambda_{k}(\theta) \frac{d}{dt}\lambda_{k+1}(\theta(t))\big|_{t=0^-}- \lambda_{k+1}(\theta) \frac{d}{dt}\lambda_{k}(\theta(t))\Big|_{t=0^-}\right]
\end{equation}
is the greatest eigenvalue of $P_{k}$ on $E_{k}(\theta) \otimes E_{k+1}(\theta)$, and
\begin{equation}
\lambda_{k}(\theta)^{2} \frac{d}{dt}\frac{\lambda_{k+1}(\theta(t))}{\lambda_{k}(\theta(t))}\Big|_{t=0^+}
=\left[\lambda_{k}(\theta) \frac{d}{dt}\lambda_{k+1}(\theta(t))\Big|_{t=0^+}- \lambda_{k+1}(\theta(t)) \frac{d}{dt}\lambda_{k}(\theta(t))\Big|_{t=0^+}\right]
\end{equation}
is the least eigenvalue of $P_{k}$ on $E_{k}(\theta) \otimes E_{k+1}(\theta)$. Hence, the criticality of $\theta$ for the ratio functional $\frac{\lambda_{k+1}}{\lambda_{k}}$ is equivalent to the fact that $P_{k}$ admits eigenvalues of both signs, which is equivalent to the indefiniteness of $\widetilde{Q_f}$ on $ E_{k}(\theta) \otimes E_{k+1}(\theta)$. The proof is complete.
\end{proof}
\begin{proposition}\label{pro 6.3}
Let $M$ be a compact embeddable strictly pseudoconvex CR manifold. For any pseudohermitian structure $\theta$ on M, the following conditions are equivalent.
\begin{enumerate}[(i)]
\item For all $f \in \mathcal{A}_{0}(M,\theta)$, the Hermitian form $\widetilde{Q_f}$ is indefinite on $E_k(\theta) \otimes E_{k+1}(\theta)$.
\item There exist finite families $\{\psi_1, \psi_2, \dots , \psi_d\} \subset E_{k}(\theta)$
and $\{\phi_1, \phi_2, \dots , \phi_e\} \subset E_{k+1}(\theta)$ of eigenfunctions such that
\begin{equation}\label{e:6.2}
\sum_{j=1}^{d} L(\psi_j) 
= 
\sum_{l=1}^{e} L(\phi_l).
\end{equation}
Here $L$ is defined by~\eqref{e:Ldef}.
\end{enumerate}
\end{proposition}
\begin{proof} The proof use similar arguments as in \cite{AhSa} and \cite{ADS2}. For the implication
$(i)\Rightarrow (ii)$, recall that for each $k$, the convex cone $\mathcal{C}_k \subset L^2(M,\theta ; \mathbb{R})$ are defined by 
\begin{equation*}
\mathcal{C}_k:=
\left\{
\sum_{j\in J}\, L(\psi_j)
\colon
\psi_j \in E_k(\theta) , J \subset \mathbb{N}, \ J \ \text{finite}
\right\}.
\end{equation*}
It suffices to prove that $\mathcal{C}_{k}$ and $\mathcal{C}_{k+1}$ have a nontrivial intersection. Indeed, if otherwise, by the classical separation theorem, there exists a real-valued function $h\in L^2(M,\theta ;\mathbb{R})$ such that
\begin{equation}
\int_M h w_1 > 0, \quad w_1 \in \mathcal{C}_{k},
\end{equation}
and 
\begin{equation}
\int_M h w_2 \leq 0, \quad w_2 \in \mathcal{C}_{k+1}.
\end{equation}
Therefore, $Q_{h}(v, v) > 0$ for all $v\in E_k$ and $Q_{h}(w, w) \leq 0$ for all $w\in E_{k+1}$. Hence 
\begin{equation}
\widetilde{Q_f}\big(v\otimes w, v \otimes w) > 0.
\end{equation}
This contradicts the assumption that $\widetilde{Q_f}$ is indefinite on $E_k(\theta) \otimes E_{k+1}(\theta)$.

Conversely, if $\psi_j$ and $\phi_l$ are as above such that \eqref{e:6.2} holds, taking integral both sides and using integration by parts, we have 
\begin{equation}
\lambda_k(\theta)\sum_{j=1}^{d} \int_M |\psi_j |^2
=
\lambda_{k+1}(\theta)\sum_{l=1}^{e} \int_M | \phi_l |^2.
\end{equation}
On the other hand, \eqref{e:6.2} also implies that for any smooth function $f$,
\begin{equation}
\sum_{j=1}^{d} Q_{f}(\psi_j, \psi_j) 
=
\sum_{l=1}^{e} Q_{f}(\phi_l, \psi_l).
\end{equation}
Therefore,
\begin{align}
\sum_{j,l}\widetilde{Q_f}\big(\psi_j \otimes \phi_l, \psi_j \otimes \phi_l) = 0
\end{align}
and thus $\widetilde{Q_f}$ is indefinite, as desired.
\end{proof}
Combining Theorem~\ref{theo 61} and Proposition~\ref{pro 6.3}, we obtain the following corollary.
\begin{corollary}
Let $M$ be a compact embeddable strictly pseudoconvex
CR manifold. Then for $k\geq 1$, a pseudohermitian structure $\theta$ on $M$ is critical for the functional $\lambda_{k+1}/\lambda_{k}$ if and only if there exist finite families $\{\psi_1, \psi_2, \dots , \psi_d\} \subset E_{k}(\theta)$
and $\{\phi_1, \phi_2, \dots , \phi_e\} \subset E_{k+1}(\theta)$ of eigenfunctions such that
\begin{equation}\label{e:6.2b}
\sum_{j=1}^{d} L(\psi_j) 
= 
\sum_{l=1}^{e} L(\phi_l),
\end{equation}
where $L$ is given by \eqref{e:Ldef}. 
\end{corollary}
\section{Examples}\label{sec:examples}
In this section, we give explicit examples of critical pseudohermitian structures for eigenvalue functionals 
of the Kohn Laplacian. If $M\subset \mathbb{C}^{n+1}$ is a compact strictly pseudoconvex real hypersurface defined by $\rho=0$ with $\rho$ is strictly plurisubharmonic and $\theta:= i\bar{\partial} \rho$ is a pseudohermitian structure on $M$, then the Kohn Laplacian is given by \cite{DLL}
\begin{equation}\label{e:klf}
\Box_b f
= 
\left(|\partial \rho|_{i\partial\bar{\partial}\rho}^{-2} \rho^{k}\rho^{\bar{j}} -\rho^{\bar{j} k}\right) f_{\bar{j} k } + n |\partial \rho|_{i\partial\bar{\partial}\rho}^{-2} \rho^{\bar{k}} f_{\bar{k}},
\end{equation}
where $f$ is a smooth function on $M$, extended smoothly to a neighborhood of $M$ in $\mathbb{C}^{n+1}$. We have used the notations $\rho_j = \partial \rho/\partial z^j$, $\rho_{j\bar{k}} = \partial^2\rho/\partial \bar{z}^{k} \partial z^j$, $\rho^{j\bar{k}}$ is the transpose of the inverse of $\rho_{j\bar{k}}$, $\rho^{\bar{j}} = \rho^{\bar{j}k}\rho_k$ (summation convention), and $|\partial \rho|_{i\partial\bar{\partial}\rho}^2= \rho^{\bar{j}k}\rho_k \rho_{\bar{j}}$. We shall also use the following observation to verify the condition \eqref{e:l1} in Theorem~\ref{thm:main2}.
\begin{lemma}\label{lem:71}
Suppose that $\psi_1,\dots , \psi_d$ are eigenfunctions with eigenvalue $\lambda_k$ which satisfy
\begin{equation}
\varphi: = \sum_{j=1}^d |\psi_j|^2 = \text{constant}.
\end{equation}
If either $\psi_j$'s are real-valued for all $j$, or the conjugations $\overline{\psi}_j$'s are CR for all $j$, then \eqref{e:l1} holds.
\end{lemma}
\begin{proof}
By direct calculations, we have that
\begin{equation}\label{e:a}
\Box_b \varphi 
=
\lambda_k \varphi + \sum_{j=1}^d \psi_j \Box_b \bar{\psi}_j - \sum_{j=1}^d |\partial_b \psi_j|^2 - \sum_{j=1}^d |\bar{\partial}_b \psi_j|^2.
\end{equation}
If $\varphi$ is a constant and $\psi_j$'s are real-valued, \eqref{e:a} implies that $\sum_{j=1}^d |\bar{\partial}_b \psi_j|^2$ is constant and hence \eqref{e:l1} follows. The argument for the case  the conjugations $\psi_j$'s are CR is similar and omitted.
\end{proof}
\begin{example}\rm 
On the sphere $\mathbb{S}^{2n+1}$ with the standard pseudohermitian structure
$\Theta: = i\partial \|Z\|^2$, the restrictions of the anti-holomorphic
functions $\bar{z}_k$ are the eigenfunctions for the first positive eigenvalue $\lambda_1=n$. Clearly, $\varphi:= \sum_{j=1}^{n+1} |\bar{z}_j|^2 =1 $ on the sphere. It follows from Lemma~\ref{lem:71} that condition \eqref{e:l1} in Theorem~\ref{thm:main2} holds. Thus, $\Theta$ is critical for $\lambda_1$-functional.
\end{example}
\begin{example}[cf. \cite{LS18}]\rm 
Consider the sphere $\mathbb{S}^3$ in $\mathbb{C}^2_{z,w}$ with the standard pseudohermitian structure. The functions $\psi_1:=\bar{z}^2, \psi_2:=\sqrt{2}\bar{z}\bar{w}$, and $\psi_3:=\bar{w}^2$ are eigenfunctions of $\lambda_2 = 2$ since $\Box_b \psi_j = 2\psi_j$ for $j=1,2,3$. On $\mathbb{S}^3$, $\sum_{j=1}^3 |\psi_j|^2 = 1$ is a constant. Therefore, by Lemma~\ref{lem:71} and Theorem~\ref{thm:main2}, the standard pseudohermitian structure on $\mathbb{S}^3$ is also critical for $\lambda_2 = 2$.

Alternatively, consider the restrictions of $\varphi_1 := \sqrt{2}\Re(z\bar{w})$, $\varphi_2 := \sqrt{2}\Im(z\bar{w})$, and $\varphi_3 := |z|^2 - |w|^2$. Then, for each $j=1,2,3$, $\varphi_j$ is a real-valued eigenfunction for $\lambda_2 = 2$. Moreover, by a direct calculation, $\sum_{j=1}^3 L(\varphi_j) = 2$ is constant.
\end{example}
\begin{example}\rm
This example generalizes the previous one.
Let $\mathcal{H}_{p,q}(\mathbb{S}^{2n+1})$ be the space of the restrictions to the sphere of the harmonic bihomogeneous polynomials of bidegree $(p,q)$. Then $\mathcal{H}_{p,q}(\mathbb{S}^{2n+1})$ is a subspace of the eigenspace $E_{\lambda}$ that corresponds to the eigenvalue $\lambda = q(p+n)$. On the other hand, the unitary group $U(n)$ acts transitively on $\mathbb{S}^{2n+1}$ and preserves the standard pseudohermitian structure. Thus, if $\{\psi_j\}$ is an orthonormal basis for this eigenspace, then 
\begin{equation}\label{e:sL}
\sum_{j}L(\psi_j) = \operatorname{constant}.
\end{equation}
When $p=0$, a similar identity holds if we take $\{\widetilde{\psi}_j\}$ to be the orthonormal basis for $\mathcal{H}_{0,q}(\mathbb{S}^{2n+1})$, since $U(n)$ also preserves the space of (anti) CR functions. Arguing similarly as in the proof of Corollary~\ref{cor:mul}, we deduce that $\sum_{j} |\widetilde{\psi}_j|^2$ is constant on the sphere. The last assertion is essentially Theorem~1 in \cite{rudin1984}.
\end{example}
\begin{example}[cf. \cite{son2018}]\label{ex:reinhardt}\rm
Let $S_r$ be the compact strictly pseudoconvex real hypersurface in $\mathbb{C}^{n+1}$ defined by $\rho = 0$, where
\begin{equation}
\rho
=
\sum_{j=1}^{n+1} \left( \log |z_j|^2\right)^2 - r^2.
\end{equation}
Then $S_r$ is the boundary of the smoothly bounded strictly pseudoconvex Reinhardt domain $\{ \rho < 0\}$. It is well-known that $S_r$ is locally homogeneous as a CR manifold, but not globally homogeneous. Thus, Corollary~\ref{cor:mul} does not apply for this case.

Let $\theta = i\partial \rho|_{S_r}$. Using \eqref{e:klf}, we can easily compute
\begin{equation}
\Box_b (\log |z_j|^2) = \frac{n}{2r^2} \log |z_j|^2.
\end{equation}
Therefore, the functions $v_j: = r^{-1}\log |z_j|^2$, $j= 1,2,\dots, n+1$, are real-valued eigenfunctions of $\Box_{b}$ that correspond to the eigenvalue $\lambda = n/(2r^2)$. Clearly, $v_1^2 + v_2^2 + \cdots + v_{n+1}^2 = 1$ on $S_r$. 
Thus, by Lemma~\ref{lem:4} and Theorem~\ref{thm:main2}, $\theta$ is critical for the eigenvalue $\lambda = n/(2r^2)$ on $S_r$.

Note in passing that since $v_j$'s are also eigenfunctions for the sub-Laplacian that correspond to the eigenvalue $\lambda(\Delta_b) = n/r^2$. By previous result \cite{ADS2}, $\theta$ is also critical for the eigenvalue of the sub-Laplacian and hence the map $(z_1, z_2, \dots, z_{n+1}) \mapsto ( \log |z_1|^2,\dots ,  \log |z_{n+1}|^2)$ is a pseudoharmonic submersion onto the sphere. 
\end{example}

\end{document}